\documentclass[12pt]{amsart}
\usepackage{amsmath, amscd}
\usepackage{amsfonts}
\usepackage{amssymb}
\usepackage{amsthm}
\usepackage{upref}
\usepackage{multirow}
\usepackage{graphicx}
\usepackage{subfloat}
\usepackage{subfig}
\usepackage{url}
\usepackage{mathtools}
\usepackage{pinlabel}
\usepackage{color}
\usepackage{hyperref}

\newcommand\quotient[2]{\raise1ex\hbox{$#1$}\Big/\lower1ex\hbox{$#2$}}
\newcommand{\gspan}[1]{\left\langle{#1}\right\rangle}
\newcommand{\scl}{\textnormal{scl}}
\newcommand{\cl}{\textnormal{cl}}
\newcommand{\sedges}{$\sigma$-edges}
\newcommand{\sloops}{$\sigma$-loops}
\newcommand{\tedges}{$\tau$-edges}

\newcommand{\Z}{\mathbb{Z}}
\newcommand{\R}{\mathbb{R}}
\newtheorem{thm}{Theorem}[section]
\newtheorem{cor}[thm]{Corollary}
\newtheorem{lem}[thm]{Lemma}
\theoremstyle{remark}
\newtheorem*{rem}{Remark}
\theoremstyle{definition}
\newtheorem{defn}[thm]{Definition}
\theoremstyle{proposition}
\newtheorem{prop}[thm]{Proposition}
\theoremstyle{question}
\newtheorem{ques}[thm]{Question}
\theoremstyle{notation}
\newtheorem{note}[thm]{Notation}
\title{Stable Commutator Length in Amalgamated Free Products}
\author{Tim Susse}
\address{The Graduate Center, City University of New York}
\email{tsusse@gc.cuny.edu}
\begin{document}
\bibliographystyle{amsalpha}

\begin{abstract}
We show that stable commutator length is rational on free products of free abelian groups amalgamated over $\mathbb{Z}^k$, a class of groups containing the fundamental groups of all torus knot complements. We consider a geometric model for these groups and parameterize all surfaces with specified boundary mapping to this space. Using this work we provide a topological algorithm to compute stable commutator length in these groups. Further, we use the methods developed to show that in free products of cyclic groups the stable commutator length of a fixed varies quasirationally in the orders of the free factors.

\end{abstract}
\maketitle

\section{Introduction}
Given a group $G$, the commutator subgroup, $[G,G]$, is the subgroup generated by all commutators: $[g ,h]=ghg^{-1}h^{-1}$, where $g ,h\in G$. For $g\in[G, G]$ its \emph{commutator length} is the minimal number of commutators needed to write $g$. In other words, $\cl_G(g)$ is the word length of $g$ in $[G, G]$. Culler showed in \cite{Culler} that, even in the free group, commutator length can behave in unexpected ways. In particular he showed that $\cl_{F_2}([a,b]^3)\le2$ by deriving the Culler identity: $$[a,b]^3=[aba^{-1},b^{-1}aba^{-2}][b^{-1}ab, b^2].$$ 

Commutator length is not stable under taking powers. Thus,  the \textit{stable commutator length} of $g\in[G, G]$ is defined as $$\scl_G(g)=\displaystyle\lim_{n\to\infty} \displaystyle\frac{\cl_G(g^n)}{n}.$$ Since commutator length is subadditive, this limit always exists, and further $\scl_G(g^n)=n\cdot\scl_G(g)$. Where there is no chance for confusion we will drop the subscript $G$. 

We can interpret commutator and stable commutator length as a ``rational covering genus" for free homotopy classes of loops in a topological space. Let $X$ be a space with $\pi_1(X)=G$. Then, $\cl_G(h)\le g$ if and only if there exists a map $f\colon S\to X$ with $f(\partial S)$ in the homotopy class $h$, where $S$ is a surface of genus $g$. This can be seen by noting that in a surface $S$ with genus $g$ and one boundary component any representative of the free homotopy class of $\partial S$ in $\pi_1(S)$ is a product of $g$ commutators.

\begin{defn}\label{admissible} A map $f\colon S\to X$ of a surface $S$ is called \emph{admissible} for $h\in [G,G]$ if for some representative $\gamma: S^1\to X$ of $h$, we have:
$$\begin{CD}
\partial S @>i>> S\\
@VV\partial fV	@VVfV\\
S^1@>\gamma>> X
\end{CD}$$
such that $\left(\partial f\right)_*\big( [\partial(S)] \big)=n(S)[S^1]$ and the diagram commutes up to homotopy.
\end{defn}

It is shown in \cite[Proposition 2.10]{SCL} that:

\begin{equation}\label{surface}\scl_G(h)=\displaystyle\inf\left\{\displaystyle\frac{-\chi^-(S)}{2n(S)}: f\colon S\to X \textnormal{ is admissible for } h\right\},\end{equation} where $-\chi^-(S)$ is defined as the sum of the Euler characteristics of all components with negative Euler characteristic.

By considering surfaces $S$ with an arbitrary number of boundary components, it is simple to extend our definition of $\scl$ to any formal linear combination of elements of $G$ which is trivial in $H_1(G)$. These chains are the one boundaries of the group and it is convenient to consider $B_1^H(G; \mathbb{R})$, the vector space of one boundaries over $\mathbb{R}$ with the relations $g^n-ng=0$ and $hgh^{-1}-g=0$ for all $g,h\in G$ and $n\in\mathbb{Z}$.

In most cases the range of values $\scl$ takes on remains unknown. The only cases which have been previously handled where $\scl$ is not identically zero are:
\begin{enumerate} \item free groups \cite{Free}; 
\item free products of abelian groups \cite{SSS, Walker:freecyclic};
\item elements of the universal central extension of $Homeo^+(S^1)$ inside of $Homeo^+(\mathbb{R})$ \cite[Chapter 2]{SCL};
\item Stein-Thompson Groups \cite{Zhuang:irrational}. 
\end{enumerate}
Calegari showed in \cite{Free} that stable commutator length is rational in any free group, and he provided an algorithm to compute $\scl$ on any finite dimensional rational subspace of one boundaries of $F_r$, showing that it is is a piecewise rational linear norm on $B_1^H(F_r)$. Generalizing that, Calegari also showed the analogous result in \cite{SSS} for all free products of abelian groups and similarly that there is an algorithm to do the computation on any finite dimensional rational subspace of the one boundaries. Groups such as these, where $\scl$ is piecewise rational linear on $B_1^H(G)$  are called \emph{PQL} --- pronounced ``pickle".

Calegari has also conjectured that the fundamental groups of all compact 3-manifolds are PQL. In this paper we will prove the following, which gives a positive answer to Calegari's conjecture in the special case of torus knot complements:\\

\noindent\textbf{Theorem \ref{AFP}.} \textit{Let $G=A \ast_{\mathbb{Z}^k} B$, where $A$ and $B$ are free abelian groups of rank at least $k$. Then stable commutator length is a piecewise rational linear function on $B_1^H(G)$, and furthermore $\scl$ is algorithmically computable on rational chains.}\\

Further, using our methods we prove the following, analogous theorem for collections of free abelian groups amalgamated over a shared subgroup.\\

\noindent\textbf{Theorem \ref{MAFP}.} \textit{Let $\left\{A_i\right\}$ be a collection of free abelian groups of rank at least k. Then stable commutator length is a piecewise rational linear function on $B_1^H(\ast_{\mathbb{Z}^k}A_i)$, where the $A_i$ are amalgamated over a common shared subgroup. Further, $\scl$ is algorithmically computable on rational chains.}\\

Previously, Fujiwara proved in \cite{QHAFP} that the second bounded cohomology of these groups (which is closely related to stable commutator length by the Bavard duality theorem -- see \cite[Theorem 2.70]{SCL}) is infinite dimensional. Generalizations of this were achieved by Bestvina and Fujiwara in \cite{BCMCG} and Calegari and Fujiwara provided the explicit lower bound $\frac{1}{312}$ for stable commutator length in these groups in \cite{SCLhyp} for words which are not conjugate to their inverses; as well as lower bounds for word hyperbolic groups and some pseudo-Anosov elements of the mapping class group.

The class of groups in Theorems \ref{AFP} and \ref{MAFP} contains many interesting groups, which arise naturally as central extensions by $\mathbb{Z}^k$ of the groups with torsion studied in \cite{SSS} and more recently in \cite{Walker:freecyclic}. When $k=1$ the class of groups is already very rich: they are central extensions of groups of the form $$H=\left(\mathbb{Z}^n\times\quotient{\mathbb{Z}}{r\mathbb{Z}}\right)\ast\left(\mathbb{Z}^m\times\quotient{\mathbb{Z}}{s\mathbb{Z}}\right).$$ In the special case where $m=n=0$ and $(r,s)=1$ we have the class of all torus knot groups \cite{Hatcher}, as noted above, which were previously known to be PQL \cite[Proposition 4.30]{SCL} using the Bavard duality theorem, as opposed to topological methods.

Using our methods, and applying results of Calegari and Walker from \cite{CalWa:quasipoly}, we also analyze the dynamics of stable commutator length of a single word, as we change the group. By exploring the relationship between groups in Theorems \ref{AFP} and \ref{MAFP} and free products of cyclic groups, we obtain the following.\\

\noindent\textbf{Corollary \ref{qrcyclic}.} \textit{Let $G=\quotient{\Z}{p_1\Z}\ast\cdots\ast\quotient{\Z}{p_k\Z}$ and $w\in F_k$. Then for $p_i$ sufficiently large, $\scl_G(w)$ depends quasirationally on the $p_i$}\\

Previously, Calegari-Walker proved in \cite{CalWa:quasipoly} that stable commutator length in surgery families is quasirational in the parameter $p$. To do this, they showed in subspaces of $B_1^H(F_r)$ spanned by $k$ surgery families $\{w_i(p)\}_{i=1}^k$ that the $\scl$ unit norm ball quasiconverges in those subspaces, and that the vertices depend quasirationally on $p$, modulo identification. The proof requires the use of the algorithm from \cite{SSS}, which is similar to the algorithm presented in section 3.

Additionally, Corollary \ref{quasirationalscl} gives a partial answer to a question of Walker from \cite{Walker:freecyclic}.

\subsection*{Acknowledgements}
I would like to thank Jason Behrstock for his helpful edits and encouragement regarding the contents in this paper. I would also like to thank Danny Calegari and Alden Walker for their interest in this project and for enlightening conversations about it.

\section{Decomposing Surfaces}
Let $A=\gspan{a_1, \ldots, a_n}$ and $B=\gspan{b_1, \ldots, b_m}$ be free abelian groups and let $G=A\ast_{\mathbb{Z}^k} B$. Up to isomorphism, we can assume that $$G=\gspan{a_1, \ldots, a_n, b_1, \ldots, b_m \mid a_1^{r_1}=b_1^{s_1}, \ldots, a_k^{r_k}=b_k^{s_k}, [a_i,a_j]=[b_p,b_q]=1}.$$

Let $X$ be the graph of spaces associated to this amalgamated free product consisting of an $n$-torus (which we will call $T_A$) and an $m$-torus (which we will call $T_B$) connected by a cylinder $T^k\times [0,1]$ where $T^k\times \{0\} \subset A$ represents the free abelian subgroup of $A$ generated by $a_1^{r_1}, \ldots, a_k^{r_k}$ and similarly for $T^k\times\{1\}\subset B$. Fix $\eta \in B_1^H(G)$. Let $f:S\to X$ be a map of a surface to $X$ with the property that $S$ is admissible for $\eta$

We will cut $X$ into pieces along the connecting cylinder and analyze how $S$ maps to each piece. Since $\eta$ is homologically trivial, by replacing occurrences of $a_i$ with $b_i^{s_i}a_i^{1-r_i}$ we can write $\eta$ in a form where the sum of the exponents of each $a_i$ and $b_i$ is zero (note that $H_1(G;\mathbb{Z}) = \gspan{a_i, b_i\mid r_ia_i-s_ib_i=0}$). We can additionally require than any summand of $\eta$ which is conjugate into either $A$ or $B$ is written in only $a_i$'s or $b_i$'s, respecitvely. We call such a representation a \textit{normal form} of $w$. It will be necessary in section 3 that $\eta$ is in this form.

\begin{defn} We say that a loop $\gamma$ representing any $g\in G$ is \emph{tight} if $\gamma\cap T^k\times(0,1)$ is an arc of the form $\{p\}\times(0,1)$ for some fixed $p\in T^k$. \end{defn}

Given $S$ and $f$ as above, we can homotope $f$ so that $f(\partial S)$ is a union of tight loops. Further, $f$ can be homotoped so that $f(S)$ intersects the loop $C=T^k\times\{\frac{1}{2}\}$ transversely, \textit{i.e.}, crossing any arc in $f^{-1}(C)\subset S$ results in switching components of $X\setminus C$.

As in \cite{SSS} consider $S\setminus f^{-1}(C)$. It is necessary to understand the surfaces mapping to each component with tight boundary. Assume that some component of $S\setminus f^{-1}(C)$ is not planar: then we would have a surface with boundary mapping to either $T_A$ or $T_B$ (which has fundamental group either $A$ or $B$, both free abelian --- without loss of generality, assume it is $A$). This map cannot be $\pi_1$-injective. If $S$ has boundary, then the sum of the boundary components must be trivial in $A$, and thus we can replace the component with a vanKampen diagram for the sum --- a planar surface. Thus, we may assume that each component of $S\setminus f^{-1}(C)$ is planar.

In the case where $G=\gspan{a, b\mid a^p=b^q}$ and $p, q$ are coprime we obtain the fundamental group of a torus knot complement. We can realize the set-up described above by replacing $T_A$ and $T_B$ with the interior of genus 1 handlebodies in a Heegaard splitting of $S^3$ and the adjoining cylinder with $A=T^2\setminus K$, which is an annulus.

Using the constructions above we prove the following classical theorem of Waldhausen. Recall that a map of a surface $S\to M$ is \emph{ incompressible} if it is $\pi_1$-injective.

\begin{thm}[Waldhausen \cite{Waldhausen}] Let $M=S^3\setminus K$, where $K$ is a torus knot. If $S$ is a closed, embedded, incompressible surface, then $S$ is a boundary parallel torus.\end{thm}

\begin{proof} Let $f:S\to M$ be the embedding of the surface $S$ and let $T$ be the torus boundary of the genus 1 Heegaard splitting of $S^3$ with $K\subset T$. Isotope $f$ so that it is transverse to $A=T\setminus K$ and $f(S)\cap A$ has a minimal number of components. We cut $S$ along $f^{-1}(A)$. As noted above, since $S$ is incompressible each component must be planar.  If any component of $S\setminus f^{-1}(A)$ has more than two boundary components, since each component of $M\setminus A$ is the interior of a genus 1 handle body we obtain a homomorphism from a free group on at least two generators to $\mathbb{Z}$, which cannot be injective. Thus, each component of $S\setminus f^{-1}(A)$ is an annulus and $S$ is a torus. 

Further, $f(S)\cap A$ is an embedded collection of curves in $A$ parallel to $K$ and the surface connects these curves from above and below. Let $C$ be some component of $S\setminus f^{-1}(A)$. Since the number of components of $f(S)\cap A$ is minimal, the boundary components of $C$ split $T$ into two pieces, exactly one of which is enclosed by $C$ and contains $K$. Locally, $C$ ``jumps" over strands of $K$. If $C$ jumps more than one strand, some arc in $\partial C$ would be (locally) enclosed by $C$, but this is impossible if $C$ is embedded. See Figure \ref{Waldhausen} for a diagram illustrating this. Thus, $S$ is boundary parallel.

\begin{figure}[h]

\centering
\labellist

\pinlabel $T$ at 70 150
\pinlabel {\color{blue} Self-intersection} at 330 410
\pinlabel {\color{black} $K$} at 485 115
\pinlabel {\color{red}$f(S)\cap A$} [l] at 425 10

\endlabellist
\includegraphics[scale=0.50]{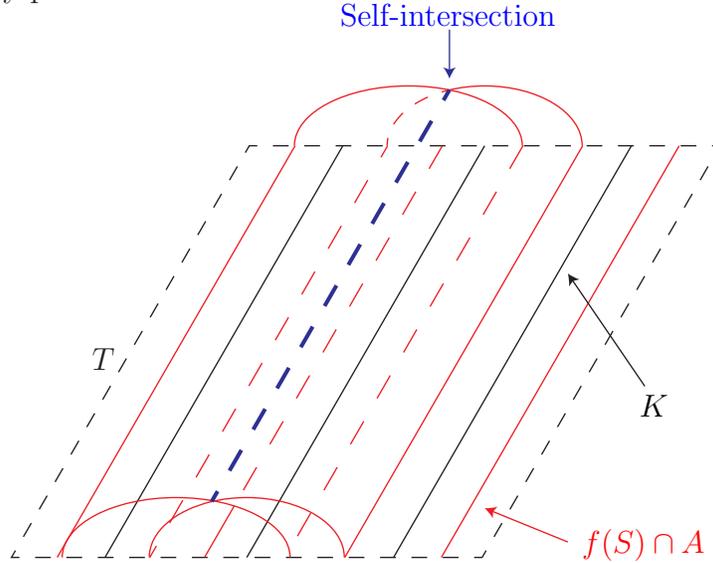}
\caption{A component jumping over more than one strand of a torus knot.}
\label{Waldhausen}
\end{figure}

\end{proof}

$f^{-1}(C)$ is necessarily a properly embedded one-submanfold of $S$. Thus it is a disjoint collection of arcs with endpoints on $\partial S$ and loops in the interior of $S$. It is possible that $f^{-1}(C)$  contains isotopic copies of a curve. In this case, depending on whether the number of copies is odd or even the number of copies can be reduced to either one or zero by cutting out all of the annuli formed by the isotopies and gluing together the new boundary components. Either the two sides of the remaining curve map to opposite components of $X\setminus C$ or the same. In the first case, we keep the remaining curve. In the second, since the map is no longer transverse to $C$, homotope the map to eliminate the loop.

The outline of the proof of Theorem \ref{AFP} is as follows:
\begin{enumerate}
\item Prove that we can assume all loop components of $f^{-1}(C)$ are separating using gluing equations (Proposition \ref{noloop});
\item Parameterize the space of all surfaces using a rational polyhedron;
\item Prove that it is possible to estimate the Euler characteristic of components of $S\setminus f^{-1}(C)$ using a piecewise rational function (Lemma \ref{approx});
\item Given compatible components, show that it is possible to glue them together and solve the gluing equations from step 1 -- we do this by adding additional loops (Lemma \ref{canglue});
\item Show that the Euler characteristic estimate from the third step is still a good estimate, even after adding loops in the fourth step (Lemma \ref{glueapprox});
\item Use linear programming to minimize $-\chi^-(S)$.
\end{enumerate}
 
\subsection{Loop Components}
The following proposition is the main goal of this section, and a formal statement of the first step in the outline above.

\begin{prop} \label{noloop} Let $S$ be a surface and $f:S\to X$ a continuous map. Then, up to replacing $S$ by a surface with higher Euler characteristic, all loop components of $f^{-1}(C)$ are separating. Further, if there are $l$ loop components, then the chain $f_*(\partial S)$ is split into at least $(l+1)$ subchains by the loops, each belonging to $\mbox{span}\left\{a_1^{r_1}, \ldots, a_k^{r_k}\right\}\subseteq H_1(G;\mathbb{R}).$\end{prop}

The boundary components of the closures of a component of $S\setminus f^{-1}(C)$ come in two types, which alternate. Using the notation of \cite{SSS}, components of $f^{-1}(C)$ will be called \emph{\sedges}, which alternate with \emph{\tedges}, which come from $\partial S$. For convenience we will distinguish between components of $f^{-1}(C)$ depending other where they are properly embedded arcs or simple closed curves. We will call the latter \emph{\sloops}. Further, boundary components  of $S$ have two forms. Either:\begin{enumerate}
\item the image is contained entirely in the torus $T_A$ or $T_B$, which we will call \textit{Abelian} loops;
\item the image alternates between tight loops in $T_A$ and tight loops in $T_B$.
\end{enumerate}

In order to consider the Abelian loops, pretend that there is a $\sigma$-edge on each Abelian loop called a \emph{dummy} edge; any other is called \emph{genuine}. 
 
Consider each component of $S\setminus f^{-1}(C)$. The division between $\sigma$- and $\tau$-edges gives a polygonal structure on each planar surface. From a collection of pieces, it must be possible to determine the Euler characteristic of $S$. To accommodate the gluings, it is necessary to use an Euler characteristic with corners, which is simply a version of an orbifold Euler characteristic. The following formula comes from thinking about each $\sigma$-edge intersecting $\partial S$ at a right angle.

\begin{defn} Given a surface with a polygonal structure on its boundary, let $c(S)$ be the number of corners. Its \textit{orbifold Euler characteristic} is given by: $$\chi_o(S)=\chi(S)-\displaystyle\frac{c(S)}{4}.$$\end{defn}

On a component of $S\setminus f^{-1}(C)$, the sum of all the \tedges, need not be zero in homology; it must, however, be an element of $\gspan{a_1^{r_1}, \ldots, a_k^{r_k}}$. This homology deficit must be made up in the genuine \sedges, since $f$ describes how the image of the boundary bounds a surface in $X$, so the sum of the images of all the $\sigma$- and $\tau$- edges in a component must be zero. Since $f(\partial S)$ is tight, each vertex maps to the same point, thus there is an element of $\mathbb{Z}^k$ attached to each genuine $\sigma$-edge and $\sigma$-loop describing how it wraps around $C$. We think of this element is a ``\emph{generalized winding number}" for the edge.

Encoding this information, label the \sedges\ $1, \ldots n$ and let \\$l_i=(l_{i,1}, \ldots, l_{i, k})$ be the element of $\mathbb{Z}^k$ attached to the $\sigma$-edge $i$. Further, let $l_{C_j}$ be the element attached to the loop components of $f^{-1}(C)$. Each component of $S\setminus f^{-1}(C)$ describes an equation in homology. In particular, the sum of all the \sedges\, and \sloops\, in a component must cancel out the sum of all the \tedges.

There are two systems of linear equations, one for components mapping to the torus $T_A$ and one for components mapping to the torus $T_B$. In each system each $\sigma$-edge is used only once, and so appears in exactly one equation. Since gluings are orientation reversing, it appears with the coefficient $1$ in the system for $T_A$ and $-1$ in $T_B$. Further, if $w$ is in a normal form, adding up all of the equations for either system gives $\sum l_i + \sum l_{C_j}=0$, since the sum of all \tedges\ mapping to $T_A$ is zero, and similarly for $T_B$. 

Proposition \ref{noloop} is thus a statement about integral solutions of this system. The following lemma guarantees us integral solutions to the system and is the technical step needed to prove that proposition.

\begin{lem} \label{matrix}
Consider a linear system of equations as follows $$\begin{bmatrix} &\Large{M_A}&\\ & &\\&\Large{M_B}& \end{bmatrix} v= \begin{bmatrix} &\Large{N_1}&\\& & \\&\Large{N_2}&\end{bmatrix},$$
such that all entries of $M_A$ and $M_B$ are either 0 or 1; each column in $M_A$ and each column in $M_B$ has a 1 in exactly one place; $N_1$ and $N_2$ are integral vectors and the system is consistent.
Letting $A_j$ be a row from $M_A$ and $B_i$ a row from $M_B$, it is possible to row reduce the matrix so that each fully reduced row is of the form 
$$\displaystyle\sum\epsilon_{i} B_{i} - \displaystyle\sum\delta_{j} A_{j},$$ where $\epsilon_i, \delta_j\in\{0,1\}$. Consequently, the system has an integral solution. \end{lem}
We provide an example of a matrix of the form described in the lemma, which illustrates the proof of the lemma in the general case.

$$\begin{bmatrix*}[r]
1 & 0 & 0 & 0 & 1 & 0 & 1 & 1 & 0 & 0   \\
0 & 1 & 0 & 0 & 0 & 1 & 0 & 0 & 0 & 0   \\
0 & 0 & 1 & 0 & 0 & 0 & 0 & 0 & 1 & 0   \\
0 & 0 & 0 & 1 & 0 & 0 & 0 & 0 & 0 & 1   \\
1 & 1 & 0 & 0 & 1 & 0 & 1 & 0 & 0 & 1   \\
0 & 0 & 1 & 0 & 0 & 1 & 0 & 1 & 0 & 0   \\
0 & 0 & 0 & 1 & 0 & 0 & 0 & 0 & 1 & 0
\end{bmatrix*}$$

\vspace{0.3cm}

Note that in the notation of Lemma \ref{matrix}, $M_A$ is the first four rows, and $M_B$ is the final three. Row reduce by subtracting the first two rows of $M_A$ from the first row of $M_B$ and proceed downwards, the third row of $M_A$ from the second row of $M_B$ and the fourth row of $M_A$ from the third row of $M_B$. This leaves us with the matrix: $$\begin{bmatrix*}[r]
1 & 0 & 0 & 0 & 1 & 0 & 1 & 1 & 0 & 0   \\
0 & 1 & 0 & 0 & 0 & 1 & 0 & 0 & 0 & 0   \\
0 & 0 & 1 & 0 & 0 & 0 & 0 & 0 & 1 & 0   \\
0 & 0 & 0 & 1 & 0 & 0 & 0 & 0 & 0 & 1   \\
0 & 0 & 0 & 0 & 0 & -1& 0 & -1 & 0 & 1\\
0 & 0 & 0 & 0 & 0 & 1 & 0 & 1 & -1 & 0\\
0 & 0 & 0 & 0 & 0 & 0 & 0 & 0 & 1 & -1
\end{bmatrix*}$$

\vspace{0.3cm}

To finish the reduction, add the first row of $M_B$ to the second, then that row to the third, and finish with the row-reduced matrix $$\begin{bmatrix*}[r]
1 & 0 & 0 & 0 & 1 & 0 & 1 & 1 & 0 & 0   \\
0 & 1 & 0 & 0 & 0 & 1 & 0 & 0 & 0 & 0   \\
0 & 0 & 1 & 0 & 0 & 0 & 0 & 0 & 1 & 0   \\
0 & 0 & 0 & 1 & 0 & 0 & 0 & 0 & 0 & 1   \\
0 & 0 & 0 & 0 & 0 & -1& 0 & -1 & 0 & 1\\
0 & 0 & 0 & 0 & 0 & 0 & 0 & 0 & -1 & 0\\
0 & 0 & 0 & 0 & 0 & 0 & 0 & 0 & 0 & -1
\end{bmatrix*}$$

\vspace{0.3cm}

Notice that no row of $M_B$ was used more than once during the reduction, the full proof relies on the fact that this holds generally.

\begin{proof}[Proof of Lemma \ref{matrix}]
Put $M_A$ is in row-echelon form by switching columns in the matrix. Let $m$ be the number of rows of $M_A$, arrange that the left-most $m\times m$ block of $A$ is an $m\times m$ identity matrix. By switching the rows of $M_B$, put $M_B$ in row-echelon form (although it will not necessarily begin with an identity matrix).

Denote the (original) rows of $M_A$ as $A_j$ and the rows of $M_B$ as $B_i$. No reductions need to be done on the matrix $A$, since it is already in row-echelon form. To reduce a row of $M_B$, first subtract off rows of $M_A$ that share $1$'s with it in the first $m$-columns. Note that since there is only one 1 in each column of $M_B$ and the first $m\times m$ block of $M_A$ is an $m\times m$ identity matrix, if $B_k$ and $B_l$ are reduced by the same $A_j$, then $k=l$.

Now, the first $m$ columns of $M_B$ are eliminated and the reduced matrix $M_B'$ has entries $1$, $0$ and $-1$. Further, each column of $M_B'$ has at most one $1$ and at most one $-1$.  Continue to reduce $M_B'$ by adding together rows (after possibly making some row switches). Let $n$ be the number of rows in $M_B$. $B_1'$ requires no reduction, as its first possible non-zero entry is in the $(m+1)$-st column, after being reduced by the rows of $M_A$. 

If $B_2'$ needs to be reduced, this is done by adding the row $B_1'=B_1''$. Call each fully reduced row $B_i''$. We claim that $B_i'$ does not get reduced by $B_k''$ if $B_k''$ was used to reduce a previous row. Let's say that $B_i'$ and $B_l'$ are both reduced by $B_k''$. Then $B_i'$ and $B_l'$ must both have an entry of $1$ or $-1$ in the left-most entry of $B_k''$, a contradiction to each column having at most one $1$ and one $-1$. Thus, each reduced row is achieved as a sum $$B_k''=\displaystyle\sum_{i=1}^n\epsilon_i B_i - \displaystyle\sum_{j=1}^m \delta_j A_j,$$ where $\epsilon_i, \delta_j\in \{0, 1\}$, since each $B_i'$ is the row $B_i$ minus a sum of rows from $M_A$. In this way, the reduced matrix has entries $1$, $0$, $-1$. Since the original system is consistent and the image vector has integral entries, this is also true of the reduced system, and so there is an integral solution.
\end{proof}

Recall that every component of $S\setminus f^{-1}(C)$ produces an equation -- thus each row of the matrices $M_A$ and $M_B$ corresponds to some component. From the first statement of Lemma \ref{matrix}, we can interpret taking a sum of rows in $M_B$ and $M_A$ as taking a union of the components. In particular, if a row in $M_B$ and a row in $M_A$ have a column where they are both non-zero, then they share a $\sigma$-edge, and in their union we glue along the shared edge. Thus, the lemma says that when reduce the matrix, \textit{each edge is used exactly once.} Further, every reduced row represents some set of glued components (along \sedges) and the remaining $1$'s and $-1$'s represent \sedges\, in the boundary of that.

\begin{proof}[Proof of Proposition \ref{noloop}]
The lemma in the case $k=1$ is equivalent to saying that no row reduction in the proof of Lemma \ref{matrix} will produce an equation $l_C=n$, for $n\neq 0$, or else it is possible to set $l_C=0$ and still get a solution to the system. Since row reductions are geometric, by the comments above, this is the same as saying that there is some union of components mapping to $T_B$ and $T_A$ that results in a surface whose only remaining $\sigma$-edge is our loop. Thus, the loop is separating, and the chain represented by the boundary components from S in this surface plus $a_1^{r_1n}$ must be zero. Thus the sum of those boundary components is in the span of $a_1^{r_1}$ in $H_1(G;\mathbb{R})$.

To prove the general case, repeat the row reduction operations $k$-times, one for each component of the solution vectors (which are in $\mathbb{Z}^k$).

Up to this point, it was implicitly assumed $S$ is connected. If not, repeat the above procedure on each component of $S$ separately to complete the proof.
\end{proof}

\subsection{Parameterizing Surfaces}
We now proceed to parameterize all surfaces $S$, with $f:S\to X$ continuous and $f_*(\partial S)$ in some finite dimensional subspace of $B_1^H(G)$.
\begin{defn} Let $T(A)$ be the set of all \tedges\ coming from $A$.
Similarly, let $T_2(A)$ be the set of ordered pairs of elements of $T(A)$, except whenever $i$ is an abelian loop the only ordered pair including $i$ is $(i, i)$.\end{defn}

Note that elements of $T(A)$ correspond to maximal subwords of components of a chain $\eta$ (or more generally a basis element of the finite dimensional subspace) coming from $A$. Further, $T_2(A)$ corresponds to all possible \sedges\  and the condition on abelian loops corresponds to the addition of dummy edges. We think of an element $(v, w)\in T_2(A)$ as a $\sigma$-edge that begins at the $\tau$-edge $v$ and ends at the $\tau$-edge $w$.

Let $C(A)$ be the real vector space spanned by $T(A)$ and $C_2(A)$ the real vector space spanned by $T_2(A)$. Given a surface $S$ and a continuous map $f:S\to X$, let $v(S)$ be the sum (in $C_2(A)$) of all of the \sedges\ in $S$. We say that $v(S)$ \emph{parameterizes} $S$.

There are two natural maps on $C_2(A)$: $$\partial: C_2(A)\to C(A)$$ which is defined on $T_2(A)$ by $\partial(a, b)=a-b$ and also $$h: C_2(A)\to A\otimes \mathbb{R}$$ defined on $T_2(A)$ by $h(a, b)=\frac{1}{2}(a+b)$. In particular $h(v)$ represents the sum of all of the \tedges\ in the real first homology of $A$. (If a surface $S$ comes from the cutting a surface with boundary along \sedges, then by putting the boundary in normal form we can see that $h(v)=0$.)

In order for a boundary component to close, it must be that $\partial(v)=0$. In particular, every $\tau$-edge appears as a beginning of exactly one $\sigma$-edge and also the end of exactly one, but the two need not necessarily be distinct.

Let $V_A$ be the set of all vectors in $C_2(A)$ such that $h(v)=0$ and $\partial(v)=0$, with all components non-negative. Since it is defined by finitely many equations and inequalities, $V_A$ is the cone on a finite sided rational polyhedron.

The next lemma shows the every integral vector $v\in V_A$ parameterizes a planar surface mapping into $T_A$. Form a graph $\Gamma$ with vertices corresponding to elements of $T(A)$ and edges corresponding to elements of $T_2(A)$ so that the edge identified with $(a,b)$ is a directed edge from the vertex $a$ to the vertex $b$. A vector $v\in V_A$ gives weights on the edges of the graph. Let $|\Gamma(v)|$ be the number of connected components of the graph weighted by $v$, after we throw out edges with zero weight. 

\begin{lem}\label{boundary} For any integral vector $v \in V_A$ there is a planar surface $S$ and a map $f:S\to T_A$ so that the vector in $C_2(A)$ coming from $S$ is $v$ and the number of boundary components of $S$ is $|\Gamma(v)|$. Further, any surface giving rise to the vector $v$ has at least $|\Gamma(v)|$ boundary components.\end{lem}

Lemma \ref{boundary} is a restatement of \cite[Lemma 3.4]{SSS}, and by setting all of the generalized winding numbers to zero the proof follows in the same way.

Let $|v|$ be the sum of the components of $v$ coming from genuine \sedges. There are exactly $2|v|$ corners on a surface, $S$, parameterized by $v$, thus $$\chi_o(S)=\chi(S)-\displaystyle\frac{|v|}{2}=2-\mbox{(\# of boundary components of $S$)}-\displaystyle\frac{|v|}{2}.$$

Noticing that $|\Gamma(nv)|=|\Gamma(v)|$, for any planar component $S_1$ of $S$, parameterized by $v_1$  with negative Euler characteristic and any $\epsilon>0$ there is an integer $n$ so that there is a surface parameterized by $nv_1$ (which we call $nS_1$) with $-\chi(nS_1)/n<\epsilon$. Thus, components  of $S$ with negative Euler characteristic are projectively negligible and contribute only corners to $\chi_o$. To compute scl we need to study those components with positive Euler characteristic, \textit{i.e.}, discs. It is here that our case differs dramatically from the case of free products.

In the case of an amalgamated free product, there exist surfaces with components that are parameterized by vectors not in $V_A$ (it may be that $h(v)\neq0$, but $\partial(v)=0$); however, for any surface obtained by cutting along \sedges, the sum of all of the components is in $V_A$. The following example shows the importance of the winding numbers on \sedges\, and will motivate the expansion of our cone $V_A$.

\textit{Example.} Let $G=\mathbb{Z}^2\ast_\mathbb{Z}\mathbb{Z}^2$ where the first $\mathbb{Z}^2$ is generated by $a$ and $b$, and the last is generated by $c$ and $d$ and we amalgamate along the shared subgroup $b=c$. Let $w=[b, d]$, which is clearly the identity. However, this is a word in a normal form, so we would like to find the disc that a tight representative of $w$ bounds.

Using the notation for the vector spaces above, we have 4 \tedges, $b$, $B$, $d$ and $D$, using the convention $B=b^{-1}$. Consider the components mapping to the torus representing the first $\mathbb{Z}^2$. $T(A)$ has dimension $2$ and $T_2(A)$ is generated by the \sedges\ $(b, b)$, $(b, B)$, $(B, b)$ and $(B, B)$. Label these \sedges\ $e_1$, $e_2$, $e_3$ and $e_4$ and the coordinates of each vector by $(v_1, v_2, v_3, v_4)$. We can easily see that $$\ker(\partial)=\mbox{span}\left\{(1,0,0,0), (0,1,1,0),(0,0,0,1)\right\},$$  $$\ker(h)=\mbox{span}\left\{(1,0,0,1),(0,1,0,0),(0,0,1,0)\right\}.$$ Consider the vector $(1, 0, 0, 1)\in V_A$. \\

\vspace{-0.1cm}

\begin{figure}[h!]
\centering
\subfloat[Annulus]{\label{annulus}
\labellist
\pinlabel $b$ at 83 116
\pinlabel $B$ at 83 91
\pinlabel $0$ at 83 52
\pinlabel $0$ at 83 29
\endlabellist
\includegraphics[scale=0.75]{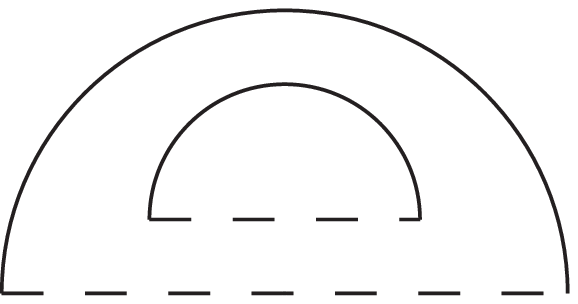}}~ \hspace{2cm}
\subfloat[Discs]{\label{discs}
\labellist
\pinlabel $b$ at 62 200
\pinlabel $-1$ at 62 135
\pinlabel $B$ at 62 91
\pinlabel$+1$ at 62 31
\endlabellist
\includegraphics[scale=0.75]{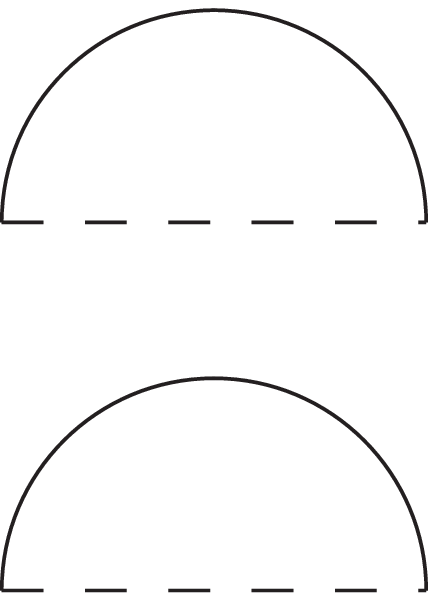}}
\caption{Representations of $(1, 0, 0, 1)$ as Planar Surfaces}
\label{aandd}
\end{figure}

Both an annulus and two discs are parameterized by $v$, as shown in Figure \ref{aandd}. In each diagram, we must attach integers to each $\sigma$-edge, as noted in the lead up to Lemma \ref{matrix}: in Figure \ref{annulus} we label each $\sigma$-edge 0, while in Figure \ref{discs} we label one $\sigma$-edge $1$ and the other $-1$. We see this by realizing $(1, 0, 0, 1) = (1, 0, 0, 0) + (0, 0, 0, 1)$ which, while not in $V_A$, are both represented by discs whose sum is in $V_A$.

\begin{figure} [h!]
\centering
\labellist
\pinlabel $d$ at 140 157
\pinlabel $b$ at 8 84
\pinlabel $-1$ at 68 84
\pinlabel $+1$ at 208 84
\pinlabel $B$ at 268 84
\pinlabel $D$ at 140 13
\endlabellist
\includegraphics[scale=0.80]{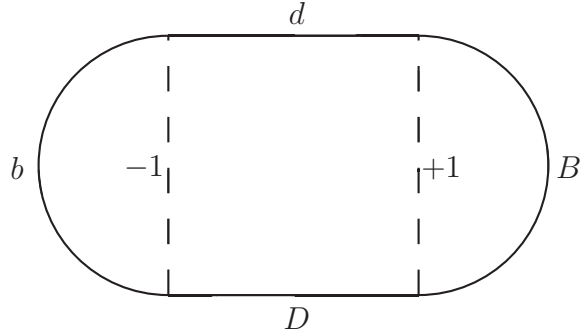}
\caption{The disc bounding the word $bdBD$}
\label{glue}

\end{figure}

We can now create the disc, as promised, by using the vector in $V_B$ given by $(d, D)+ (D, d)$, as shown in Figure~\ref{glue}. Each $\sigma$-edge is drawn in the figure with dashed lines and is labelled with and integer indicating its class in $H_1(C)$, while each $\tau$-edge is labeled by the maximal subword it represents. This disc confirms that $[b, d]=1$.

\subsection{Disc Vectors} We call any integral vector $v$ which is realized by a map of a disc into $T_A$ a \emph{disc vector}. From Lemma \ref{boundary}, we conclude that these are exactly the vectors $v$ for which $|\Gamma(v)|=1$. The example shows, however, that we need to look outside of $V_A$ to find all of the necessary disc vectors to compute scl. To do this effectively, we need to expand our vector space $C_2(A)$.

\begin{note}Let $\overline{C_2(A)}=C_2(A)\times \mbox{span}_\mathbb{R}\{a_1^{r_1}, \ldots a_k^{r_k}\}$. Abusing notation, we consider the element $((v,w),(l_1, \ldots, l_k))$ of $\overline{C_2(A)}$  as \linebreak $(v,w)+\sum l_ia_i^{r_i}$. We extend the maps $\partial$ and $h$ to $\overline{C_2(A)}$ as follows:
 \begin{center}
 $\partial((v, w))=v-w$ and $\partial(a_i^{r_i})=0$\\
 $h((v, w))=\frac{1}{2}\left(v+w\right)$ and $h(a_i^{r_1})=-r_ia_i$.\end{center}
 Further, let $\overline{V_A}$ be $\ker(\partial)\cap\ker(h)$ intersected with the non-negative orthant of $\overline{C_2(A)}$\end{note}

Notice that the projection $\pi:\overline{V_A} \to C_2(A)$ is one-to one, and so it can be identified its image. Since $\overline{V_A}$ is a cone on a finite-sided rational polyhedron in $\overline{C_2(A)}$, the same holds in $C_2(A)$. In fact, there is a map $\pi(\overline{V_A})\to \overline{V_A}$ given by $v\mapsto (v, h(v))$. Further, $V_A\subseteq\overline{V_A}$ in this way. Alternatively, we can identify $V_A$ with $V_A\times\{0\}$. Let $\mathcal{D}_A$ be the set of disc vectors in $\overline{V_A}$. We look for representations of $v\in V_A$ by the most disc vectors \textit{i.e.}, representations of the form $v=\sum t_iv_i + v'$ such that $v_i\in\mathcal{D}_A$, $t_i>0$ and $v'\in\overline{V_A}$. We call these \textit{acceptable representations}. Following \cite{SSS} we make the following definition.

\begin{defn} The \emph{Klein function} on $V_A$ is defined by $\kappa(v)=\max\sum t_i$, where the maximum is taken over all acceptable representations $v=\sum t_iv_i + v'$.\end{defn}

As in Calegari's algorithm, we use $\kappa$ to approximate the Euler characteristic of a vector. However, $\kappa$ will always consider any dummy $\sigma$-edge corresponding to an abelian loop representing an element in the shared $\mathbb{Z}^k$ subgroup as a disc. In light of this, we define $$\chi_o(v)=\kappa(v) - \displaystyle\frac{|v|}{2}-v_{ab},$$ as the orbifold Euler characteristic of a \textit{vector}, where $v_{ab}$ is the $l_1$-norm of the component of $v$ coming from dummy \sedges\ corresponding to abelian loops labelled with an element of the amalgamating subgroup.

\begin{lem} \label{approx} Let $S_A$ be a surface with $v(S_A)=v$, then $\chi_o(v)\ge \chi_o(S_A)$. Further, for any rational vector $v\in V_A$ and $\epsilon>0$ there exists a surface $S_A$ with $v(S_A)=nv$ and $\chi_o(v)\le\chi_o(S_A)/n +\epsilon$.\end{lem}
\begin{proof} This first statement follows from the fact that every disc component has Euler characteristic 1. Thus, disc components of $S_A$ can add at most $\kappa(v)-v_{ab}$ to the Euler characteristic, since $\kappa$ will always count those dummy \sedges\ as loops. All other components contribute zero or a negative number to the Euler characteristic along with their corners, so $\chi(S_A)$ is at most $\kappa(v)-v_{ab}-\frac{|v|}{2}$ and the result follows.

For the second, let $v=\sum t_iv_i +v'$ be an acceptable representation of $v$ that realizes $\kappa(v)$. By perturbing the $t_i$ slightly, we can assume each is rational, and their sum is within $\frac{\epsilon}{2}$ of $\kappa(v)$. Choose $n$ so that: $nv=n_iv_i + v''$, where $n_i\in\mathbb{Z}$ and $v''$ is an integral vector. Now, form a surface using $n_i$ discs parameterized by $v_i$, when $v_i$ is not a dummy $\sigma$-edge representing an element of the amalgamating subgroup, and additional components (which are not discs) corresponding to $v''$. If $v''$ is non-zero, then we add all abelian loops labelled by elements of $\gspan{a_1^{r_1}, \ldots a_k^{r_k}}$ to one of the components in the surface parameterized by $v''$. If $v''=0$, then we are forced to add the abelian loops into a disc.

Now take covers so that the Euler characteristic of all non-disc components are negligible. In particular if $S_A'$ is the union of those components, find $m$ so that $\chi(S_A')/m>-\frac{\epsilon}{2}$. If $v''=0$ but $v_{ab}\neq 0$, then there is some disc vector $w$ in our surface, so that the corresponding disc is not realized (since the component will contain the abelian loop). In this case, create $m$ discs corresponding to $w$ in $mv$, one of which contains the abelian loop (and so is not realized as a disc). As a result, we obtain a surface $S_A$ whose Euler characteristic (divided by its covering degree) is as close to $\kappa(v)-v_{ab}$ as we would like. 
\end{proof}

The following lemmas are identical to \cite[Lemma 3.10, Lemma 3.11]{SSS}, and the proof of the first is identical if we replace $V_A$ by $\overline{V_A}$.

\begin{lem}\label{linear} $\kappa(v)$ is a non-negative, concave function on $\overline{V_A}$ that is linear on rays and $\kappa=1$ exactly on the boundary of $\mbox{conv}(\mathcal{D}_A+\overline{V_A})$ inside $\overline{V_A}$.\end{lem}

\begin{lem}\label{poly}The sets $\mbox{conv}(\mathcal{D}_A)$ and $\mbox{conv}(\mathcal{D}_A+\overline{V_A})$ are finite sided, convex, rational, closed polytopes, whose vertices are elements of $\mathcal{D}_A$.\end{lem}

\begin{proof} By \cite[Lemma 4.8]{SSS} the faces of $\ker(\partial)$ are in one-to-one correspondence with recurrent subgraphs of $\Gamma$, so that the disc vectors in $\overline{V_A}$ are contained in the open faces corresponding to recurrent connected subgraphs of $\Gamma$. If $F$ is one of them, we know that $\mbox{conv}(F\cap\mathcal{D}_A)$ is a finite sided polytope (see \cite{Sails} for an explanation). Now, there are only finitely many connected recurrent subgraphs of $\Gamma$, so only finitely many of these faces to consider. The the convex hull is also a finite sided polytope. Similarly, \cite{Sails} implies that $\mbox{conv}(\mathcal{D}_A+\overline{V_A})$ is covered by only finitely many translates of $\overline{V_A}$ (a slight generalization of Dickson's Lemma), implying the result.\end{proof}

\begin{cor}\label{kappa} The functions $\kappa(v)$ and $\chi_o(v)$ are each equal to the minimum of the finite set of rational linear functions.\end{cor}
\begin{proof} Note that any function that is linear on rays and exactly 1 on the boundary of $\mbox{conv}(\mathcal{D}_A+\overline{V_A})$ must be equal to $\kappa(v)$. As noted in \cite{CLO}, there is a primitive inward pointing normal vector $v_j$ for each codimension 1 face $F_j$ of $\mbox{conv}(\mathcal{D}_A+\overline{V_A})$, so that $\gspan{v, v_j}=a_j$ for all $v\in F_j$, and for all other vectors in the polytope, $\gspan{v, v_j}>a_j$. Set $\kappa'(v)=\min_j\left\{\gspan{v, \displaystyle\frac{v_j}{a_j}}\right\}$. This function clearly meets the criteria of Lemma \ref{linear}, and so it must be equal to $\kappa(v)$. Further, each of the functions in the minimum is linear and there are only finitely many, since by Lemma \ref{poly} $\mbox{conv}(\mathcal{D}_A+\overline{V_A})$ has only finitely many sides. \end{proof}

\section{Computing SCL}
Let $A$ and $B$ be two free abelian groups of rank at least $k$ and let $G=A\ast_{\mathbb{Z}^k}B$ and consider a chain $\eta\in B_1^H(G)$. From the methods of the previous section there are two polytopes $V_A$ and $V_B$, each of which has  an associated piecewise rational linear function, denoted by $\chi_o^A$ and $\chi_o^B$, respectively. Consider $V_A\times V_B$. We say two \sedges\ $v=(\tau_1, \tau _2)$ and $w=(\tau_1', \tau_2 ')$ are \textit{compatible} if $\tau_1$ is followed by $\tau_2'$ and $\tau_1'$ is followed by $\tau_2$ in $\eta$, considering each word in $\eta$ cyclically. Two vectors $v\in V_A$ and $w\in V_B$ are \textit{compatible} if their components from compatible \sedges\ are equal.

To complete the proof of Theorem \ref{AFP} we first need to prove that given compatible vectors it is possible to construct a surface admissible for $\eta$. We will prove in Lemma \ref{canglue} that we can accomplish this by adding \sloops. We then simplify the surface, as in Proposition \ref{noloop}, and show in Lemma \ref{doubling} and Lemma \ref{glueapprox} that we can control the added \sloops\ in a strong sense.

\begin{lem}\label{canglue}Given two compatible vectors $v\in V_A$ and $w\in V_B$ there exists a surface parameterized by $v$ and $w$, formed by adding \sloops, which is admissible for $\eta$.\end{lem}
\begin{proof}
Since each $\sigma$-edge carries a weight, and weights attached to glued \sedges\ must be opposite, we need to guarantee a solution to the gluing equations from section 2. We will do that by adding \sloops\ to components coming from $T_A$ and $T_B$ and gluing those accordingly.

For each component coming from $T_A$, if the sum of its \tedges\ is $\sum m_ia_i^{r_i}$, add $\sum |m_i|$ loops to the component each labelled with $\pm a_i^{r_i}$ and similarly for $T_B$. We can glue loops labelled with $\pm a_i^{r_i}$ to loops labelled $\mp b_i^{s_i}$. 

If $T_A$ has more loops labelled with $\pm a_i^{r_i}$ than $T_B$ then glue as many loops as possible. Since $w$ is in normal form, the remaining loops contain the same number of $a_i^{r_i}$ labels as $-a_i^{r_i}$ labels. Choose a component of $S$ coming from $T_B$ and add as many loops as there are remaining in $T_A$. Label half with $b_i^{s_i}$ and half with $-b_i^{s_i}$ and complete the gluing. See Figure \ref{LoopDiagram} for an example.

\begin{figure}
\labellist
\pinlabel $a^3$ at 70 335
\pinlabel $-1$ at 65 263
\pinlabel $a$ at 70 192
\pinlabel $A^2$ at 70 132
\pinlabel $+1$ at 65 59
\pinlabel $A^2$ at 70 14
\pinlabel $B^2$ at 349 329
\pinlabel $b^2$ at 349 296
\pinlabel $-1$ at 283 246
\pinlabel $+1$ at 405 246
\pinlabel $b$ at 344 127
\pinlabel $B$ at 344 90
\endlabellist
\includegraphics[scale=0.50]{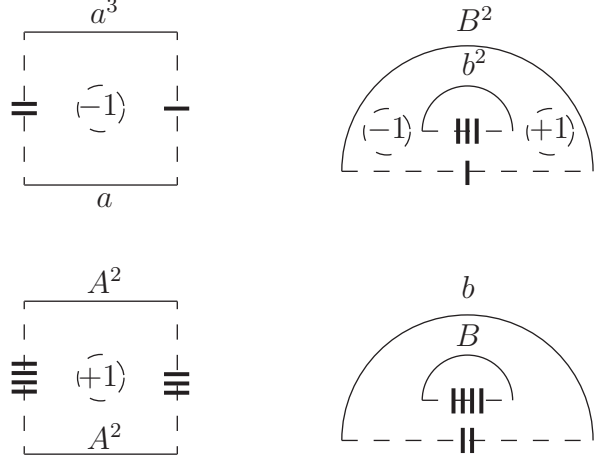}
\caption{A surface bounding the chain $a^3B^2ab+A^2BA^2b^2$ in $\gspan{a, b\mid a^4=b^3}$. All \sedges\ are labelled zero and we glue each $+1$ loop to a $-1$ loop.}
\label{LoopDiagram}
\end{figure}
By labeling all of the remaining \sedges\ 0, we obtain the required gluing.\end{proof}

After adding \sloops\ in the proof of the proposition we have altered the Euler characteristic of the resulting surface. In what follows we will simplify this surface using Proposition \ref{noloop}, by altering the winding numbers on \sedges\ and \sloops, to bring the Euler characteristic of the resulting surface in line with the estimate obtained in Proposition \ref{approx}.

First, if $\eta$ is either a word or does not contain a subchain in the span of $\{a_1^{r_1},\ldots a_k^{r_k}\}$, Proposition \ref{noloop} guarantees an integral solution to the gluing equations with all loops eliminated, since there are no subchains with sum in the span of $\{a_1^{r_1}, \ldots, a_k^{r_k}\}$. Further, Lemma \ref{approx} guarantees that $\chi_o^A+\chi_o^B$ is an estimate for the Euler characteristic of this surface.

If $\eta$ contains a subchain in the span of the amalgamating words, then Proposition \ref{noloop} implies that the gluing equations have a solution, leaving a limited number of separating loops. The following proposition shows us that we can strongly control the propagation of loops when taking covers in the sense of Lemma \ref{approx}.

\begin{prop}\label{doubling} Let $v\in V_A$ and $w\in V_B$ be compatible integral vectors. Let $S$ be a surface parameterized by the pair with the least number of loop components. Suppose that $S$ contains $k$ loop components. Then there is a surface $S'$ parameterized by $2v$ and $2w$ which contains at most $k$ loop components.\end{prop}

\begin{proof} Take $v$, $w$ and $S$ as above. According to Proposition \ref{noloop}, since there are $k$ loop components $\partial S$ is divided into $(k+1)$ subcollections $\partial_1, \ldots, \partial_{k+1}$ so that each subcollection is mapped to a chain over $G$ which is in the span of the amalgamating words. Further, no component of $S\setminus f^{-1}(C)$ can contain \tedges\ from two of the subcollections and no boundary component of $\partial_i$ can be joined to one in $\partial_j$ by an arc in $S$ which does not cross a loop component of $f^{-1}(C)$.

Form $S'$ as in Lemma \ref{approx}, so that $\partial S'$ is a double cover of $\partial S$. The surface $S'$ may have more boundary components than $S$, meaning that some boundary components may be doubled in $S'$. There are two possible cases.

If some component with a $\tau$-edge from $\partial_i$ does not correspond to a disc vector, then in $S'$ each lift of a boundary component in $\partial_i$ can be connected to the other lift of that boundary component by an arc which remains in one component. Thus, in $S'$ the lift of $\partial_i$ cannot be separated into two subcollections.

If every component with \tedges\ in $\partial_i$ is in $\mathcal{D}_A\cup\mathcal{D}_B$, consider the most na\"ive gluing of the \sedges, where we duplicate the gluing pattern from $S$. The two lifts of each boundary component of $\partial_i$ give two boundary components of $S'$. There are necessarily two subcollections, which we will call $\partial_i$ and $\partial_i'$. 

From the naive gluing since each $\sigma$-edge appears twice, we can switch the pairing. In this way, $\partial_i$ is joined to $\partial_i'$ by an arc which does not cross a loop. Thus $\partial_i$ and $\partial_i'$ are not disconnected in this surface and can therefore not be separated into two subcollections.

Thus, in $S'$ we can arrange for there to be $k+1$ subcollections of boundary components which can be pairwise separated by loops. Thus only at most $k$ loops to solve the gluing equations, as desired.
\end{proof}

\begin{prop}\label{glueapprox} Let $v$ and $w$ be compatible rational vectors, then for any $\epsilon>0$ there exists a positive integer $N$ and a surface $S$ parameterized by $Nv$ and $Nw$ so that $$\chi_o^A(v)+\chi_o^B(w)-\displaystyle\frac{\chi(S)}{N}<\epsilon.$$\end{prop}
\begin{proof} This follows from Lemma \ref{approx} and Proposition \ref{doubling}
\end{proof}

Proposition \ref{glueapprox} tells us  the Euler characteristic of a surface an be approximated by the piecewise rational linear function $\chi_o^A+\chi_o^B$, even accounting for adding the loops needed to solve the gluing equations. Further, the condition that $S$ has boundary representing $\eta$ is linear,  so  $\chi_o^A+\chi_o^B$ can be maximized on the intersection of this linear subspace with $V_A$, a finite-sided polyhedron, using linear programming. Thus, we obtain the following theorem.

\begin{thm}\label{AFP} Let $G=A\ast_{\mathbb{Z}^k} B$, where $A$ and $B$ are free abelian groups or rank at least $k$. $G$ is PQL and there exists an algorithm to compute stable commutator length for any rational chain in $B_1^H(G)$.\end{thm}

As mentioned before, the methods above naturally extend to amalgamations of several free abelian groups as above over a single $\mathbb{Z}^k$. Let $X$ be the graph of spaces formed from, tori $T_{A_i}$, with $T_{A_i}$ connected to $T_{A_{i+1}}$ with a cylinder. Using this space, we obtain the following, more general theorem.

\begin{thm}\label{MAFP} Let $G=\ast_{\mathbb{Z}^k}A_i$, where $\{A_i\}$ is a collection of free abelian groups of rank at least $k$, and all of the $A_i$ share a common $\mathbb{Z}^k$ subgroup. Then $G$ is PQL and there exists an algorithm to compute stable commutator length for any rational chain in $B_1^H(G)$.\end{thm}

\begin{proof} Since each torus has only one set of gluing equations coming from the amalgamation, we write out the gluing system as a block matrix. Call the blocks $M_i$ and use row switches to put each one in row-echelon form. Note that, by construction, the block $M_i$ shares entries in columns with only $M_{i-1}$ and $M_{i+1}$. Call the overlap $M_{i,i-1}$ and $M_{i, i+1}$, respectively. Note that the $M_{i, i+1}$ blocks are in ``block row-echelon form". We begin row reducing the $M_2$ block, as in Lemma \ref{matrix}. At the end of this procedure, either the matrix is fully row reduced, and we can set set all loops not gluing a component from $T_{A_{1}}$ to one from $T_{A_2}$ and we are finished, or else there are reduced rows in $M_2$ with zeroes in every column of $M_{1,2}$. 

 In the first case, the \sloops\ arise only when there are collections of \tedges\ from $T_{A_1}$ and $T_{A_2}$ which cannot be connected by paths that do not cross into a component from $T_{A_3}$. This is a weak form of Proposition \ref{noloop}. 

In the second we row reduce the next block using these zero rows. Since these rows correspond to unions of components from $T_{A_1}$ and $T_{A_2}$ with only \sedges\ between $T_{A_2}$ and $T_{A_3}$ remaining. Thus,  two such reduced rows must correspond to non-overlapping sums. Thus, we can continue to apply the method of Lemma \ref{matrix} and show that loops arise only when there are \tedges\ from $T_{A_i}$ which cannot be connected by a path which stays in components from $T_{A_{i}}$ and $T_{A_{i-1}}$ or  $T_{A_{i}}$ and $T_{A_{i+1}}$ or crossing \sloops.
From this weak form of Proposition \ref{noloop}, we can still use the covering procedure of Proposition \ref{doubling} and obtain the result.
\end{proof}

\section{Applications of the Theorem}

\subsection{Examples and Formulas} While the algorithm described in the previous section is unwieldy, even in simple situations, and suspected to run in double exponential time \cite{SSS}, it is possible to do some calculations without appealing to all of the machinery described.

\begin{prop}\label{formula} Let $G=\gspan{a,b \mid a^p=b^q}$, then $$\scl([a^m, b^n])=\max\left\{\min\left\{\frac{1}{2}- \frac{m}{\mbox{lcm}(m, p)}, \frac{1}{2}-\frac{n}{\mbox{lcm}(n, q)}\right\},0\right\}$$.\end{prop}

\begin{rem}If $[a^m,b^n]=1$ in the group $G$, the stable commutator length is 0. However, the algorithm will give the answer $-\frac{1}{2}$ and produce a disc bounding the word. Thus, we must include the maximum in the formula above.\end{rem}
\begin{proof} Form the vector space $C_2(A)$ with basis $e_1=(a^m, a^m)$, $e_2=(a^m, a^{-m})$, $e_3=(a^{-m}, a^m)$ and $e_4=(a^{-m}, a^{-m})$. The boundary and homology maps tell us that for an admissable vector $\sum v_ie_i$, $v_2=v_3$ and $v_1=v_4$. It remains to determine the Klein function. Clearly, the vector $e_2+e_3$ forms a disc, and less obviously $\frac{\mbox{lcm}(m, p)}{m} e_1$ makes a disc, and similarly for $e_4$ (see Figure \ref{adisc} for an example when $m=4$, $p=6$). Thus $$\kappa(v)=\frac{m}{\mbox{lcm}(m,p)}v_1+\frac{1}{2}v_2+\frac{1}{2}v_3+\frac{m}{\mbox{lcm}(m,p)}v_4.$$
Similar formulas hold for $C_2(B)$, spanned by $f_1=(b^n, b^n)$, $f_2=(b^n, b^{-n})$ and so on, where a vector is given by $w=\sum w_if_i$. In particular:
$$\kappa(w)=\frac{n}{\mbox{lcm}(n,q)}w_1+\frac{1}{2}w_2+\frac{1}{2}w_3+\frac{n}{\mbox{lcm}(n,q)}w_4.$$
Further, the compatibility equations impose the constraints $v_1=w_3$, $v_2=w_1$, $v_3=w_4$ and $v_4=w_2$ and $|v|=|w|=2$. Thus: \begin{align*}\frac{-\chi(v,w)}{2} = & \left(\frac{1}{4}-\frac{m}{2\mbox{lcm}(m,p)}\right)v_1+\left(\frac{1}{4}-\frac{m}{2\mbox{lcm}(m,p)}\right)v_4\\ &+ \left(\frac{1}{4}-\frac{n}{2\mbox{lcm}(n,q)}\right)w_1+\left(\frac{1}{4}-\frac{n}{2\mbox{lcm}(n,q)}\right)w_4.\end{align*}
Since $v_1=v_4$ and $w_1=w_4$, the above simplifies to:
$$\frac{-\chi(v,w)}{2}=\left(\frac{1}{2}-\frac{m}{\mbox{lcm}(m,p)}\right)v_1+\left(\frac{1}{2}-\frac{n}{\mbox{lcm}(n,q)}\right)w_1.$$
It is clear the the minimum is either achieved when $v_1=1$ or when $w_1=1$, which are mutually exclusive possibilities.
\end{proof}

\begin{figure}
\centering
\labellist
\pinlabel $a^4$ at 104 187
\pinlabel $-1$ at 41 125
\pinlabel $-1$ at 156 125
\pinlabel $a^4$ at 195 46
\pinlabel $a^4$ at 5 46
\pinlabel $0$ at 102 16
\endlabellist
\includegraphics[scale=0.5]{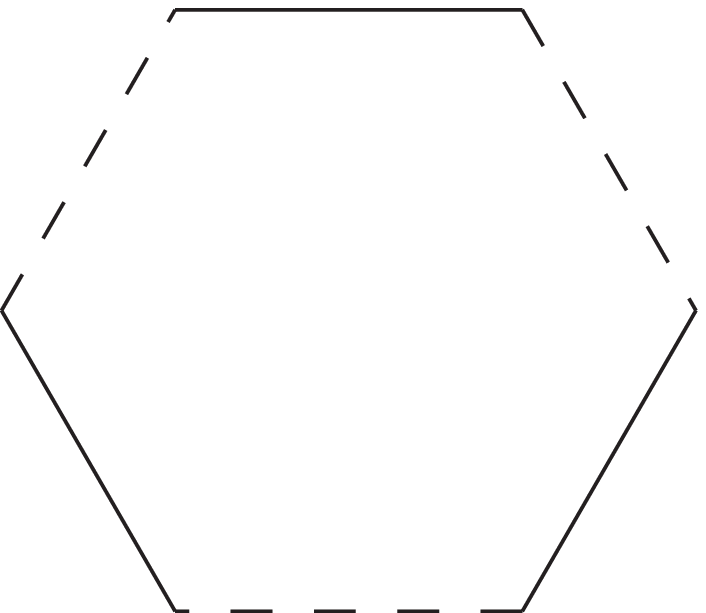}
\caption{A disc for the vector $3(a^4, a^4)$ in the case $p=6$.}
\label{adisc}
\end{figure}
This formula is similar to one that appears in \cite[Corollary 5.3]{Walker:freecyclic}, and we will show in the next section that it is, in fact, a generalization. The formula in the proposition depends only on residue of $m \pmod p$ and $n \pmod q$. This phenomenon is theoretically explained in the next section.

\subsection{Homological Considerations}
The groups $G=\mathbb{Z}^n\ast_{\mathbb{Z}^k}\mathbb{Z}^m$ are naturally central extensions of groups of the form $$H=\left(\mathbb{Z}^{n-k}\times\prod_{i=1}^k\quotient{\mathbb{Z}}{n_i\mathbb{Z}}\right) \ast \left(\mathbb{Z}^{m-k}\times\prod_{i=1}^k\quotient{\mathbb{Z}}{m_i\mathbb{Z}}\right)$$.

This central extension is of a very particular form, since $i(\mathbb{Z}^k)\cap [G,G]=\{1\}$.

\begin{prop}\label{torsion} Let $1\to A\xrightarrow{\,i\, }E\xrightarrow{\,\pi\,}G\to1$ be a central extension with $A$ a free Abelian group so that $i(A)\cap [E,E]$ is empty. Then the induced map $H^2(G;\mathbb{R})\to H^2(E;\mathbb{R})$ is injective. Thus, the Euler class of the extension is a torsion class in $H^2(G;A)$. \end{prop}

\begin{proof}
We consider the spectral sequence 
$$1\to H^1(G;\mathbb{R})\to H^1(E;\mathbb{R})\to H^1(A;\mathbb{R})\to H^2(G;\mathbb{R})\to H^2(E;\mathbb{R})$$
See \cite{McCleary:spectral} for background on spectral sequences. Since the real first cohomology of any group is simply the space of homomorphisms to $\mathbb{R}$, we need to show that the restriction map $\text{Hom}(E;\R)\to\text{Hom}(A;\R)$ is surjective (\emph{i.e}, that every homomorphism $A\to\R$ extends to a homomorphism $E\to\R$).

Let $\phi\colon A\to\R$ be a homomorphism. We will extend this to a homomorphism $\phi\colon E\to\R$. To do this, it suffices to define $\overline{\phi}$ on $\quotient{E}{[E,E]}$.

Since $i(A)\cap[E,E]=\emptyset$, the induced map $A\xrightarrow{\,i\,}E\to\quotient{E}{[E,E]}$ is injective, and $A$ injects into the free part of the abelianization of E. We will denote elements of the abelianization by $\overline{g}$. Let $A=\gspan{a_1, \ldots, a_k}$, then there exist generators $\overline{g_1},\ldots, \overline{g_k},\ldots,\overline{g_n}$ of $\quotient{E}{[E,E]}$ so that $\overline{i(a_i)}=r_i\overline{g_i}$. Define $$\overline{\phi}(\overline{g_i})=\begin{cases} \displaystyle\frac{\phi(a_i)}{r_i}, & \text{ if } 1\le i\le k\\
0, & \text{ otherwise}\end{cases}.$$
Now, $\overline{\phi}:\quotient{E}{[E,E]}\to\R$ lifts to a homomorphism $\phi:E\to\R$. Thus the map $H^1(E;\R)\to H^1(A;\R)$ is surjective.

By the exactness of the sequence, this implies that $$\text{ker}\left(H^2(G;\R)\to H^2(E;\R)\right)=0.$$

The second statement of the Proposition follows immediately from the above and the Universal Coefficient Theorem \cite[Chapter 3.A]{Hatcher}, since $H^*(G;A)=H^*(G;\mathbb{Z})\otimes A$ whenever $A$ is a torsion free.
\end{proof}
\begin{rem} Gersten proved in \cite{Gersten:bdd} that if $A=\mathbb{Z}$, and the image of a class in $H^2(G;\R)$ is bounded, it is a bounded integral class. In the same paper, he proved that extensions with bounded Euler class are quasi-isometric to $G\times\Z$.\end{rem}

\begin{rem} The groups $\mathbb{Z}^n\ast_{\mathbb{Z}^k}\mathbb{Z}^m$ are naturally quasi-isometric to $\left(\Z^{n-k}\times T\right)\ast\left(\Z^{m-k}\times T'\right)\times\Z^k$, where $T$ and $T'$ are torsion. They are not, however, virtually direct product of the form $\left(A\ast B\right)\times \Z^k$ with $A$ and $B$ free abelian groups. They are virtually of the form $\left(A\ast B \ast \Z \ast \Z \ast \cdots \ast \Z\ast\Z\right)\times\Z^k$, where there are $2k$ many $\Z$ free factors. Computationally, this is a challenging situation for the algorithm in \cite{SSS} and information is also lost when passing to finite index subgroup.

Thus, the algorithm in this paper provides a direct approach to free products of abelian groups with torsion, and can in fact detect phenomena not apparent using the algorithm in \cite{SSS}. For instance, in the group $B_3=\gspan{a, b\mid a^2=b^3}$, $\scl([a,b])=0$, which is not obvious when considering its finite index $F_2\times \mathbb{Z}$ subgroup, since $[a,b]\neq 1$. In fact, working through the algorithm, we determine that $[a,b]=[a,B]$, which is conjugate to $[b,a]=[a,b]^{-1}$, detecting a mirror in the braid group.
\end{rem}

Considering Proposition \ref{torsion}, the following is an extension of \cite[Proposition 4.30]{SCL}.

\begin{prop}\label{isometry}  Let $$1\to A\xrightarrow{\,i\,}E\xrightarrow{\,\pi\,}G\to 1$$ be a central extension so that the induced map $H^2(G;\R)\to H^2(E;\R)$ is injective. Then the projection map $E\xrightarrow{\,\pi\,} G$ is an isometry for $\scl$.\end{prop}

Thus for groups of the form $\mathbb{Z}\ast_\mathbb{Z}\mathbb{Z}$, their projections to free products of cyclic groups, as studied extensively in \cite{Walker:freecyclic} is scl preserving. In fact, it is possible to extend the above proposition to the groups handled in Theorem \ref{MAFP}. When considering the case when $A_i=\mathbb{Z}$ for all $i$, these amalgamated free products project to any free product of finite cyclic groups.

\subsection{Quasirationality in Free Products of Cyclic Groups}
Propositon \ref{isometry} allows us to study $\scl$ in free products of cyclic groups by instead looking at an amalgamated free product. Walker studied $\scl$ in free products of cyclic groups in \cite{Walker:freecyclic}, and posed the following question:

\begin{ques} Let $G=\ast \quotient{\Z}{p_i\Z}$. For a fixed word $w\in F_r$, and $\overline{w}$, its projection to $G$, is it true that $\scl(\overline{w})$ varies like a quasilinear function in $\frac{1}{p_i}$, for $p_i$ sufficiently large?\end{ques}

To analyze this question, we will ask the analogous question for amalgamated free products of free Abelian groups, as in Theorems \ref{AFP} and \ref{MAFP}. 

\begin{defn} A function $P(n)$ is called a \emph{quasipolynomial} if there is a least positive integer $\pi$ and a finite set of polynomials, $\{P_0(n), \ldots, P_{\pi-1}(n)\}$, so that for any integer $n$, with $n\equiv k \pmod{\pi}$ and $0\le k \le \pi -1$ we have that $P(n)=P_{k}(n)$.\end{defn}

In other words, $P(n)$ is a polynomial whose coefficients can vary, but depend only on the residue of $n$ modulo $\pi$. Further, $\pi$ is called the \emph{period} of $P(n)$, and its \emph{degree} is the maximal degree of $P_i(n)$, for $0\le i\le \pi-1$. For example, the formula in Proposition \ref{formula} for $\scl([a^m, b^n])$ is a quotient of two quasipolynomials in $p$ and $q$.

In this section we will prove that, fixing a word $w$, $\scl(w)$ behaves quasirationally as the amalgamated free product of free Abelian groups is changed. The following theorem is critical to the proof below.

\begin{thm} \cite[Theorem 3.5]{CalWa:quasipoly} \label{QIQvertices} For $1\le i\le k$ let $v_i(n)$ be a vector in $\mathbb{R}^d$ whose coordinates are rational functions of $n$ of size $O(n)$ and let $V_n=\{v_i(n)\}_{i=1}^k$. Let $S_n$ be the convex hull of the integer points in the convex hull of $V_n$. Then for $n>>0$, there exists some $\pi\in\mathbb{Z}$ so that the vertices of $S_{\pi n+i}$ are the columns of a matrix whose entries are integer polynomials in the variable $n$.\end{thm}

This condition is called \emph{QIQ}. It is important to be careful here because for different values of $i$, it is possible that there are different numbers of vertices, so that the matrices for each $i$ are not the same size. The process of looking at $\pi n + i$ instead of $n$ is called \emph{passing to a cycle}. It is also clear that, since these integer points are in the convex hull of $V_n$, their coordinates are also of size at most $O(n)$.

Further, if $V_n$ is instead a collection of integral linear extremal vectors of a polyhedral cone, the analogous result is true.

\begin{cor} \cite[Corollary 3.7]{CalWa:quasipoly} \label{coneQIQ} Let $V_n$ be a family of cones with integral linear extremal vectors. Then the integer hull (open or closed) of $V_n-0$ is QIQ.\end{cor}

In what follows, we consider only the case where $k=2$ (for notational ease), though the general case follows using the same arguments. Further, to emphasize the dependence on the group $\Z\ast\Z$, we will write $\overline{w}=w_{p,q}$ to mean the projection to $\gspan{a, b: a^p=b^q}$.

Consider some work $w\in F_2$. Then by Proposition \ref{isometry}, we know that $\scl(w_{p,q})$ is equal to the stable commutator length of that word in the corresponding free product of cyclic groups.

Now, for every choice of $p$ we obtain polyhedra $V_A(p)$ and $V_B(q)$ and corresponding Klein functions $\kappa_{A,p}$ and $\kappa_{B,q}$ corresponding to data about the integer hulls of the disc faces of $V_A(p)$ and $V_B(q)$. Since the $w_{p,q}$ are combinatorially identical, we can prove the following lemma, similar to a statement in \cite[Section 4.4]{SSS}, which states that the polyhedron $V_A(p)$ satisfies the criteria of Corollary \ref{coneQIQ}.

\begin{lem}\label{quasilinear} There are extremal integral vectors for $V_A(p)$ whose coordinates are quasilinear in $p$.\end{lem}
\begin{proof} This follows from the description of extremal rays given in \cite[Lemma 4.11]{SSS}. Recall from Section 4.4 that $\Gamma$ is graph whose vertices are in one-to-one corresponce with $T(A)$ and edges in one-to-one correspondence with $T_2(A)$. For an edge-path $\phi$ in $\Gamma$, we will let $h(\phi)$ be the sum of the values of $h$ on the elements of $T_2(A)$ corresponding to the edges in $\phi$. 
Then the extremal rays of $V_A$ correspond to either:
\begin{enumerate}
\item embedded oriented cycles $\phi$ in $\Gamma$ with $h(\phi)=0$;
\item $h(\phi')\phi-h(\phi)\phi'$, where $\phi, \phi'$ are distinct oriented embedded cycles in $\Gamma$ with $h(\phi')>0$ and $h(\phi)<0$;
\item a pair $\phi$ and an integer $v_\phi$ so that $h(\phi+v_\phi)=0$.
\end{enumerate}

Note that only the third option is affected by the choice of $p$. In fact, in the case of the groups $G_{p,q}$ it is possible to ignore the second option all together, since for a high enough power of any oriented embedded cycle, $h(\phi)$ will have a multiple divisible by $pa$.
The necessary multiple, and $v_\phi$ are determined quasilinearly in $p$, with period $h(\phi)$.
\end{proof}

Now, the combinatorics of $V_A(p)$ are identical for every $p$, the disc faces of $V_A(p)$ are the same and, by the previous lemma, satisfy the criteria of Corollary \ref{coneQIQ}. Thus, we can conclude that that vertices of $\mathcal{D}_A(p)$ are QIQ in $p$.

\begin{defn}\label{pqnorm} For a given $p,q\in\Z$, let $\scl_{p,q}(w)=\scl(w_{p,q})$ in the group $\Z\ast_{\Z}\Z=\gspan{a,b\mid a^p=b^q}.$ Further, denote by $\|\cdot\|_{p,q}$ the induced semi norm on $B_1^H(F_2;\R)$.\end{defn}

As in \cite{CalWa:quasipoly}, we analyze the unit  balls in the norms described in Definition \ref{pqnorm} for every $p,q$. Walker proved in \cite{Walker:freecyclic} that as $p,q\to\infty$ the norms $\|\cdot\|_{p,q}$ converge to the standard $\scl$ norm on $B_1^H(F_2;\R)$, \textit{i.e.}, the unit balls in the norms converge to the $\scl$ norm ball in the free group.

To understand this convergence we need to study the function $\kappa_{A,p}(v)$, which is the Klein function on $\text{conv}(\mathcal{D}_A(p)+V_A(p))$. Recall from Lemma \ref{poly} that the set of vertices of this polyhedron is a subset of the vertices of $\text{conv}(\mathcal{D}_A(p))$ and thus by Corollary \ref{coneQIQ} have coordinates which are quasilinear polynomials in $p$.

From here, we can determine precisely how the functions $\kappa_p(v)$ depend on $p$. For the next Theorem, it is necessary that we have a chain $w$ in normal form, as defined in Section 2. Then every surface admissible for $w$ comes from a fixed subset of $V_A$. Let $v_{a}$ be the component of $v$ corresponding the winding numbers from the amalgamation. We define:
$$V_A^{(0)}=V_A\cap \{v_{a}=0\}.$$

Note that that when varying $p$ and $q$, the polyhedral cones $V_A^{(0)}(p)$ and $V_B^{(0)}(p)$ are completely independent of $p$ and $q$. In what follows, we will drop the dependence on $p$ and $q$.

\begin{thm}\label{convergence} Let $w_1,\ldots,w_k$ be a collection of rational chains in $B_1^H(F_2;\R)$ and let $B_{p,q}$ be the unit ball in the norm $\|\cdot\|_{p.q}$. Then the vertices of $B_{p,q}$ have coordinates which are eventually quasirational in $p$ and $q$.\end{thm}
\begin{proof} 
Corollary \ref{kappa} implies that the function $\kappa_p(v)$ on $V_A^{(0)}$ varies like a quotient of two linear polynomials in $p$. 
The result then follows immediately from \cite[Theorem 4.6]{CalWa:quasipoly} by considering the linear map from compatible vectors in $V_A^{(0)}\times V_B^{(0)}$ to $B_1^H(F_2;\R)$, whose image, by construction, is precisely the subspace spanned by $w_1, \ldots, w_k$. The image of the set of vectors with $\chi(v)\le1$ is then the unit norm ball in this subspace.\end{proof}

\begin{cor}\label{quasirationalscl} For $p$ and $q$ sufficiently large $\scl_{p,q}(w)$ is a quasirational function in $p$ and $q$ for any fixed chain $w\in B_1^H(F_2;\R)$.
\end{cor}
\begin{proof}
This follows from Theorem \ref{convergence} by considering the case $k=1$.
\end{proof}

Letting $G=\ast_{\Z^k} A_i$, for $\{A_i\}_{i=1}^s$ a family of free Abelian groups of rank $r_i$, $r=\sum r_i$ and  $w\in F_r$. Then, as in the presentation for $G$ in section 2, the group is determined up to isomorphism by a finite set of parameters, $p_{i,j}$ coming from the relations $a_{1,j}^{p_{1,j}}=\cdots=a_{s,j}^{p_{s,j}}$.

\begin{cor}\label{qrMAFP}  For $p_{i,j}$ sufficiently large, $\scl_G(w)$ depends quasirationally on $p_{i,j}$. \end{cor}

Further, we obtain the following statement about free products of cyclic groups.
\begin{cor}\label{qrcyclic} Let $G=\quotient{\Z}{p_1\Z}\ast\cdots\ast\quotient{\Z}{p_k\Z}$ and $w\in F_k$. Then for $p_i$ sufficiently large, $\scl_G(w)$ depends quasirationally on the $p_i$\end{cor}

\begin{proof} The statement follows immediately from Corollary \ref{qrcyclic} by considering the projection $\ast_{\Z} \Z\to G$ and Proposition \ref{isometry}.\end{proof}

Notably, \cite{Walker:freecyclic} shows that the $\|\cdot\|_{p,q}$ norm converges to the standard $\scl$ norm for $F_2$ at least as fast as $O(\frac{1}{p}+\frac{1}{q})$. A similar statement holds for the $\scl$ norm for $F_k$. Corollary \ref{quasirationalscl} says that this is the slowest rate possible, though it is still an open question whether it is possible to obtain faster convergence.

\bibliography{Susse_SCLinAFP}

\end{document}